\documentclass[12pt]{amsart}
\usepackage{a4wide}
\usepackage[T1]{fontenc}
\usepackage{amssymb,amsmath,amsthm,latexsym}
\usepackage{mathrsfs}
\usepackage[usenames,dvipsnames]{color}
\usepackage{euscript}
\usepackage{graphicx}
\usepackage{mdwlist}
\usepackage{enumerate}
\usepackage{mathtools,dsfont,wasysym}
\usepackage{bbm}
\usepackage{stmaryrd}
\usepackage{centernot}
\usepackage{todonotes}
\usepackage{hyperref}
\hypersetup{colorlinks=true,  linkcolor=blue, citecolor=magenta, urlcolor=cyan}

\makeatletter
\@namedef{subjclassname@2020}{\textup{2020} Mathematics Subject Classification}
\makeatother

\newtheorem{theorem}{Theorem}[section]

\newtheorem{corollary}[theorem]{Corollary}

\newcounter{maintheorem}

\theoremstyle{remark}
\newtheorem{remark}[theorem]{Remark}
\theoremstyle{definition}
\newtheorem{definition}[theorem]{Definition}

\newtheorem{example}[theorem]{Example}
\numberwithin{equation}{section}
\makeatother

\newcommand{\nn}[1]{{\left\vert\kern-0.25ex\left\vert\kern-0.25ex\left\vert #1
		\right\vert\kern-0.25ex\right\vert\kern-0.25ex\right\vert}}

\renewcommand{\leq}{\leqslant}
\renewcommand{\geq}{\geqslant}
\newcommand{\dif}{\nabla_{xy}}

\newcounter{smallromans}

	{\end{list}}

\begin{document}
	\title[Nonexistence of solutions to parabolic problems on weighted graphs]{Nonexistence of solutions to parabolic problems with a potential on weighted graphs}
	
	\author[D.D. Monticelli]{Dario D. Monticelli}
	\address[D.D. Monticelli]{Politecnico di Milano, Dipartimento di Matematica, Piazza Leonardo da Vinci 32, 20133 Milano, Italy}
	\email{dario.monticelli@polimi.it}
	
	\author[F. Punzo]{Fabio Punzo}
	\address[F. Punzo]{Politecnico di Milano, Dipartimento di Matematica, Piazza Leonardo da Vinci 32, 20133 Milano, Italy}
	\email{fabio.punzo@polimi.it}
	
	\author[J. Somaglia]{Jacopo Somaglia}
	\address[J. Somaglia]{Politecnico di Milano, Dipartimento di Matematica, Piazza Leonardo da Vinci 32, 20133 Milano, Italy}
	\email{jacopo.somaglia@polimi.it}
	
	\date{\today}
	
	\begin{abstract}
		We investigate nonexistence of nontrivial nonnegative solutions to a class of semilinear parabolic equations with a positive potential, posed on weighted graphs. Assuming an upper bound on the Laplacian of the distance and a suitable weighted space-time volume growth condition, we show that no global solutions exists. We also discuss the optimality of the hypotheses, thus recovering a critical exponent phenomenon of Fujita type.
	\end{abstract}
	\maketitle
	
	%-------------------------------------------------------%
	%                                                       									%
	% 						INTRODUCTION 											%
	%                                                      											 %
	%-------------------------------------------------------%
	
	\section{Introduction}

	We investigate nonexistence of nonnegative, nontrivial global solutions of semilinear parabolic inequalities of the following type:
	\begin{equation}\label{e: mainequation}
		u_t \geq \Delta u + v u^{\sigma} \quad \text{ in } V\times (0, +\infty)\,,
	\end{equation}
	where $V$ is a weighted graph, $\Delta$ stands for the weighted Laplacian on the graph, $v$ is a given positive potential and $\sigma>1\,.$ 
	
	This subject has a long and rich history when posed not on graph, but in the Euclidean space or on Riemannian manifolds. More precisely,
	global existence and finite time blow-up of solutions have been deeply studied in $\mathbb{R}^N$; when $v\equiv 1$, it has been shown by Fujita in \cite{F}, and in \cite{H} and \cite{KST} for the critical case, that 
	\begin{enumerate}[a)]
		\item if  $1 < \sigma\le \sigma^*:=1+\frac{2}N$, then there is no nonnegative nontrivial global solutions;
		\item if $\sigma > \sigma^*$ solutions corresponding to nonnegative data that are sufficiently small in a suitable sense, are global in time.
	\end{enumerate}
	A complete account of results concerning blow-up and global existence of solutions to semilinear parabolic equations posed in $\mathbb R^N$ can be found, e.g., in \cite{BB}, \cite{DL}, \cite{Levine} and in references therein.
	
	On Riemannian manifolds, the situation can be different. Some results can be found e.g. in \cite{BPT, GMPu, Sun1, MaMoPu, Punzo, Pu22, WY2, Zhang}. Furthermore, analogue results have also been established for quasilinear evolution equations (see e.g. \cite{MaMoPu, MeGP1, MeGP2, MeMoPu, Mitidieri2}). In particular, in \cite{MaMoPu} it has been shown that the unique global nonnegative solution is $u\equiv 0$ on the manifold $M$, provided that a certain weighted space-time volume growth condition is fulfilled. In particular, for instance, in the semilinear case with $v\equiv 1$, that condition reads:
	\[\operatorname{Vol} (B_R) \leq C R^{\frac2{\sigma-1}} (\log R)^{\frac 1{\sigma-1}} \text{ for all  sufficientlt large }R>0\,;\]
	here $B_R\subset M$ denotes the geodesic ball of radius $R$ centered at some reference point $o\in M.$ Observe that both exponents are sharp. 
	
	\medskip
	
	Very recently the stationary problem corresponding to \eqref{e: mainequation}, that is 
	\begin{equation}\label{e100f}
		\Delta u + v u^{\sigma} \leq 0 \quad \text{ in } V,
	\end{equation} 
	has also been investigated (see \cite{GHS, MPS1}, and \cite{BMP, MP} for the case $\sigma=1$). More specifically in \cite{MPS1}, given a distance on the graph, an upper bound on its Laplacian and a growth condition on a suitable weighted volume of balls are assumed. Under such hypotheses, it is proved that problem \eqref{e100f} does not admit any nonnegative nontrivial solution. 
	
	\medskip
	
	Semilinear parabolic equations with $v\equiv 1$ on graphs have already been studied in the literature. Blow-up of solutions and global existence have been addressed in \cite{LW2, LWU, MeGPu3} by means of semigroup methods and heat kernel estimates, relying on specific metric or spectral assumptions on the graph. 
	
	In this paper, supposing an upper bound on the Laplacian of the distance, and an appropriate space-time weighted volume growth condition, we show that \eqref{e: mainequation} admits no nontrivial nonnegative solutions (see Theorem \ref{teo1} and its corollaries for precise statements). We also show that the result is sharp when $V=\mathbb Z^N$ and $v\equiv 1$ (see Example \ref{ex1}). Let us explicitly note that our line of proof is totally different from those of \cite{LW2, LWU, MeGPu3}. In fact, we do not use the semigroup approach or heat kernel estimates, but we derive certain a priori estimates on the solution $u$, by means of a test function argument. In doing that, we assume and exploit an upper bound on the Laplacian of the distance on the graph. 
	
	\medskip
	
	The paper is organized as follows. In Section \ref{mr} we describe the mathematical framework and we state the main results. Section \ref{sec3} is devoted to the proofs of the main theorem (i.e. Theorem \ref{teo1}) and of its corollaries, which are concerned with the case of an infinite graph $V$. In Section \ref{ex} we illustrate some examples. Lastly, in Section \ref{finite} we briefly consider the case of a finite graph $V$.
	
	\medskip
	
	\textbf{Acknowledgements.} The authors are members of GNAMPA-INdAM and are partially supported by
	GNAMPA projects 2024. Moreover, Dario D. Monticelli and F. Punzo acknowledge that this work is part of the PRIN ”Geometric-analytic methods for PDEs and applications", ref. 2022SLTHCE, financially supported by the EU, in the framework of the "Next Generation EU initiative".

	\section{Assumptions and main results}\label{mr}
	\subsection{The graph setting}
	Let $V$ be a countable set and $\mu:V\to (0,+\infty)$ be a given function.
	Furthermore, let
	\begin{equation*}
		\omega:V\times V\to [0,+\infty)
	\end{equation*}
	be a symmetric function with zero diagonal and finite sum, i.e.
	\begin{equation}\label{omega}
		\begin{array}{ll}
			\text{(i)}\,\, \displaystyle\omega_{xy}=\omega_{yx}&\text{for all}\,\,\, (x,y)\in V\times V;\\
			\text{(ii)}\,\, \displaystyle\omega_{xx}=0 \quad\quad\quad\,\, &\text{for all}\,\,\, x\in V;\\
			\text{(iii)}\,\, \displaystyle \sum_{y\in V} \omega_{xy}<\infty \quad &\text{for all}\,\,\, x\in V\,.
		\end{array}
	\end{equation}
	Thus, we define  \textit{weighted graph} the triplet $(V,\omega,\mu)$, where $\omega$ and $\mu$ are the so called \textit{edge weight} and \textit{node measure}, respectively. Observe that assumption (ii) corresponds to ask that $V$ has no loops.
	\smallskip
	
	\noindent Let $x,y$ be two points in $V$; we say that
	\begin{itemize}
		\item $x$ is {\it connected} to $y$ and we write $x\sim y$, whenever $\omega_{xy}>0$;
		\item the couple $(x,y)$ is an {\it edge} of the graph and the vertices $x,y$ are called the {\it endpoints} of the edge whenever $x\sim y$; we will denote by $E$ the set of all edges of the graph;
		\item a collection of vertices $ \{x_k\}_{k=0}^n\subset V$ is a {\it path} consisting of $n+1$ nodes if $x_k\sim x_{k+1}$ for all $k=0, \ldots, n-1.$
	\end{itemize}

	A weighted graph $(V,\omega,\mu)$  is said to be
	\begin{itemize}
		\item[(i)] {\em locally finite} if each vertex $x\in V$ has only finitely many $y\in V$ such that $x\sim y$;
		\item[(ii)] {\em connected} if, for any two distinct vertices $x,y\in V$, there exists a path joining $x$ to $y$;
		\item[(iii)] {\em undirected} if its edges do not have an orientation.
	\end{itemize}
	For any $x\in V$ we will call {\em degree of $x$} the cardinality of the set $$\{y\in V\colon y\sim x\}.$$
	
	\noindent A {\it pseudo metric} on $V$ is a symmetric map, with zero diagonal, $d:V\times V\to [0, +\infty)$, which also satisfies the triangle inequality
	\begin{equation*}
		d(x,y)\leq d(x,z)+d(z,y)\quad \text{for all}\,\,\, x,y,z\in V.
	\end{equation*}
	In general, $d$ is not a metric, since there may be points $x, y\in V$, $x\neq y$ such that $d(x,y)=0\,.$
	Finally, we define the \textit{jump size} $j>0$ of a general pseudo metric $d$ as
	\begin{equation}\label{e14f}
		j:=\sup\{d(x,y) \,:\, x,y\in V, \omega(x,y)>0\}.
	\end{equation}
	
	For any given $\Omega\subset V$ its {\it volume} is defined as
	\[\operatorname{Vol}(\Omega):=\sum_{x\in \Omega}\mu(x)\,.\]
	
	For a set $\Omega\subset V$ we denote by $\textbf{1}_\Omega$ the characteristic function of $\Omega$.

	\subsection{The weighted Laplacian} Let $\mathfrak F$ denote the set of all functions $f: V\to \mathbb R$\,. For any $f\in \mathfrak F$ and for all $x,y\in V$, let us give the following
	\begin{definition}
		Let $(V, \omega,\mu)$ be a weighted graph. For any $f\in \mathfrak F$,
		\begin{itemize}
			\item the {\em difference operator} is
			\begin{equation}\label{e2f}
				\nabla_{xy} f:= f(y)-f(x)\,;
			\end{equation}
			\item the {\em (weighted) Laplace operator} on $(V, \omega, \mu)$ is
			\begin{equation*}
				\Delta f(x):=\frac{1}{\mu(x)}\sum_{y\in V}\omega_{xy}[f(y)-f(x)]=\frac 1{\mu(x)}\sum_{y\sim x}\omega_{xy}[f(y)-f(x)]\quad \text{ for all }\, x\in V\,.
			\end{equation*}
		\end{itemize}
	\end{definition}
	
	It is straightforward to show, for any $f,g\in \mathfrak F$, the validity of
	\begin{itemize}
		\item  the {\it product rule}
		\begin{equation*}
			\nabla_{xy}(fg)=f(x) (\nabla_{xy} g) + (\nabla_{xy} f)g(y) \quad \text{ for all } x,y\in V\,;
		\end{equation*}
		\item the {\it integration by parts formula}
		\begin{equation}\label{e4f}
			\sum_{x\in V}[\Delta f(x)] g(x) \mu(x)=-\frac 1 2\sum_{x,y\in V}\omega_{xy}(\dif f)(\dif g)= \sum_{x\in V}f(x) [\Delta g(x)] \mu(x)\,,
		\end{equation}
		provided that at least one of the functions $f, g\in \mathfrak F$ has {\it finite} support.
	\end{itemize}
	
	\subsection{Main results}
	
	\begin{definition}
		We say that $u:V\times[0,+\infty)\rightarrow\mathbb{R}$ is a nonnegative \textit{very weak solution} of \eqref{e: mainequation} if $u\geq0$ and
		\begin{equation}\label{e:48}
			\begin{split}
				\int_0^{+\infty}\sum_{x\in V}\mu(x)\Delta u(x,t) \varphi(x,t)+\mu(x)v(x,t)u^\sigma(x,t)\varphi&(x,t)+\mu(x)u(x,t)\varphi_t(x,t)\,dt\\
				&+\sum_{x\in V}\mu(x)u(x,0)\varphi(x,0)\leq0
			\end{split}
		\end{equation}
		for every $\varphi:V\times[0,+\infty)\rightarrow\mathbb{R}$ such that $\varphi\geq0$, $\operatorname{supp}\varphi\subset[0,T]\times A$ for some $T>0$ and some finite set $A\subset V$ and $\varphi(x,\cdot)\in W^{1,1}([0,+\infty))$ for every $x\in V$.
	\end{definition}
	\begin{remark}
		Note that if a solution $u(x,t)$ is $W^{1,1}_{\text{loc}}([0,+\infty))$ in $t$ for every $x\in V$, then \eqref{e:48} follows from testing equation \eqref{e: mainequation} against $\varphi$, by a simple integration by parts.
	\end{remark}
	
	\medskip
	
	In this paper, we always make the following assumptions:
	\begin{equation}\label{e7f}
		\begin{aligned}
			\text{(i)}\,\,\, & (V, \omega, \mu) \text{ is a connected, locally finite, undirected, weighted graph};\\
			\text{(ii)}\,\, \, &  \text{ there exists a constant } C>0 \text{ such that for every } x\in V, \\
			&\qquad \sum_{y\sim x}\omega_{xy}\leq C \mu(x);\\
			\text{(iii)} \,\,\,& \text{there exists a \textit{pseudo metric}}\,\, d \,\,\,\text{such that its jump size $j$ is finite}; \\
			\text{(iv)}\,\,\,& \text{the ball}\,\,\, B_r(x) \,\,\,\text{with respect to}\,\,\, d\,\,\, \text{is a finite set, for any}\,\,\, x\in V,\,\,\, r>0;\\
			\text{(v)}\,\, & \text{ for some } x_0\in V,\, R_0>1,\, \alpha\in [0, 1],\, C>0 \text{ there holds } \\
			& \qquad\Delta d(x, x_0)\leq \frac C{d^\alpha(x, x_0)} \text{ for any }\, x\in V\setminus B_{R_0}(x_0).
		\end{aligned}
	\end{equation}

	\begin{remark}\label{4444}
		We note that if $(V, \omega, \mu)$ is a connected, locally finite, undirected, weighted graph such that $\sum_{y\sim x}\omega_{xy}\leq C \mu(x)$ for every $x\in V$, endowed with a metric $d$ with finite jump size $j$ and such that $B_R(x)$ is finite for every $x\in V, R>0$, then condition \eqref{e7f}--(v) is automatically satisfied with $\alpha=0$. Indeed we have
		\begin{equation*}
			\Delta d(x,x_0)=\frac{1}{\mu(x)}\sum_{y\sim x}\omega_{xy}[d(y,x_0)-d(x,x_0)]\leq \frac{1}{\mu(x)}\sum_{y\sim x}\omega_{xy} d(x,y)\leq  Cj,
		\end{equation*}
		for every $x\in V$. Then Theorem \ref{teo1} and all its corollaries immediately apply in this case, with $\alpha=0$.
		
		This is in particular the case of a connected, locally finite, undirected, weighted graph $(V, \omega, \mu)$ such that $\sum_{y\sim x}\omega_{xy}\leq C \mu(x)$ for every $x\in V$, endowed with its natural distance $d_*$, defined by
		\begin{equation}\label{333}
			d_*(x,y)=\min\{k\in\mathbb{N}\,:\,\exists\text{ path }x_0=x,x_1,\ldots,x_k=y\text{ of }k+1\text{ nodes connecting }x,y\}
		\end{equation}
		for every $x,y\in V$, which has jump size $j_*=1$.
	\end{remark}
	
	\begin{remark}\label{rem111}
		We note that given $\alpha\in[0,1]$, $R_0>2j$ we have
		\begin{equation}\label{13}
			\Delta d(x, x_0)\leq \frac C{d^\alpha(x, x_0)} \text{ for any }\, x\in V\setminus B_{R_0}(x_0).
		\end{equation}
		if and only if
		\begin{equation}\label{14}
			\Delta d^{1+\alpha}(x, x_0)\leq C\text{ for any }\, x\in V\setminus B_{R_0}(x_0).
		\end{equation}
		Indeed, by the convexity of the function $p^{1+\alpha}$, for very $p,r\geq0$ we have
		\[
		p^{1+\alpha}\geq r^{1+\alpha}+(1+\alpha)r^\alpha(p-r),
		\]
		hence, with $p=d(y,x_0)$, $r=d(x,x_0)$, we obtain
		\begin{equation*}
			\begin{split}
				\Delta d^{1+\alpha}(x,x_0)&=\frac{1}{\mu(x)}\sum_{y\sim x}\omega_{xy}(d^{1+\alpha}(y,x_0)-d^{1+\alpha}(x,x_0))\\
				&\geq\frac{1}{\mu(x)}\sum_{y\sim x}(1+\alpha)\omega_{xy}d^{\alpha}(x,x_0)(d(y,x_0)-d(x,x_0))\\
				&=(1+\alpha)d^{\alpha}(x,x_0)\Delta d(x,x_0).
			\end{split}
		\end{equation*}
		From the inequality above we see that \eqref{14} readily implies \eqref{13}, actually for every $\alpha\geq0$. On the other hand, using Taylor expansions, for every $p,r\geq0$ we have
		\[
		p^{1+\alpha}= r^{1+\alpha}+(1+\alpha)r^\alpha(p-r)+\frac{\alpha(1+\alpha)}{2}\eta^{\alpha-1}(p-r)^2
		\]
		for some $\eta$ between $p,r$. Thus, again with $p=d(y,x_0)$, $r=d(x,x_0)$, we obtain
		\begin{equation*}
			\begin{split}
				\Delta d^{1+\alpha}(x,x_0)&=\frac{1}{\mu(x)}\sum_{y\sim x}\omega_{xy}(d^{1+\alpha}(y,x_0)-d^{1+\alpha}(x,x_0))\\
				&=\frac{1}{\mu(x)}\sum_{y\sim x}(1+\alpha)\omega_{xy}d^{\alpha}(x,x_0)(d(y,x_0)-d(x,x_0))\\
				&\,\,\,\,\,+\frac{1}{\mu(x)}\sum_{y\sim x}\frac{\alpha(1+\alpha)}{2}\omega_{xy}\eta^{\alpha-1}(d(y,x_0)-d(x,x_0))^2.
			\end{split}
		\end{equation*}
		Now note that for every $y\sim x$, by the definition of jump size, we have
		\[
		d(x,x_0)-j\leq d(y,x_0)\leq d(x,x_0)+j
		\]
		and that
		\[
		\min\{d(x,x_0),d(y,x_0)\}\leq\eta\leq\max\{d(x,x_0),d(y,x_0)\}.
		\]
		Hence $\eta\geq d(x,x_0)-j$. Since $\alpha-1\leq0$, if \eqref{13} holds then for every $x\in V\setminus B_{R_0}(x_0)$ we have
		\begin{equation*}
			\begin{split}
				\Delta d^{1+\alpha}(x,x_0)&\leq(1+\alpha)d^{\alpha}(x,x_0)\Delta d(x,x_0)\\
				&\,\,\,\,\,+\frac{\alpha(1+\alpha)}{2}j^2(d(x,x_0)-j)^{\alpha-1}\frac{1}{\mu(x)}\sum_{y\sim x}\omega_{xy}\leq C,
			\end{split}
		\end{equation*}
		that is \eqref{14}.
	\end{remark}
	
	Now we state our main theorem, which deals with the case of an infinite graph $V$.
	
	\begin{theorem}\label{teo1}
		Let assumption \eqref{e7f} be satisfied. Assume the graph $(V,\omega,\mu)$ is infinite.  Let $v\colon V\times [0,\infty)\to \mathbb{R}$ be a positive function, $\sigma>1$ and $\theta_1\geq 2$, $\theta_2\geq 1$ such that $\frac{\theta_1}{\theta_2}\geq 1 + \alpha$, with $\alpha\in[0,1]$ as in \eqref{e7f}. Suppose that for every $R\geq R_0>1$
		\begin{equation}\label{e: stimavolumi}
			\int_0^{+\infty}\sum_{x\in V} \textbf{1}_{E_R}(x,t) v(x,t)^{-\frac{1}{\sigma-1}}\mu(x)dt\leq CR^{\frac{(1+\alpha)\sigma}{\sigma-1}},
		\end{equation}
		where
		\begin{equation}\label{22}
			E_R\coloneqq\{(x,t)\in V\times [0,+\infty)\colon R^{\theta_1}\leq d(x_0,x)^{\theta_1}+t^{\theta_2}\leq 2 R^{\theta_1}\}
		\end{equation}
		and  $x_0\in V$ as in \eqref{e7f}. Let $u\colon V\times [0,\infty)\to \mathbb{R}$ be a non-negative very weak solution of \eqref{e: mainequation}, then $u\equiv 0$.
	\end{theorem}
	
	Now we discuss some consequences of Theorem \ref{teo1} when the potential $v$ has a special form.
	
	\begin{corollary}\label{coroll1}
		Let assumption \eqref{e7f} be satisfied. Assume the graph $(V,\omega,\mu)$ is infinite.  Let $\sigma>1$ and $v\colon V\times [0,\infty)\to \mathbb{R}$ be a positive function such that $v(x,t)\geq f(t)g(x)$ for every $(x,t)\in V\times[0,+\infty)$, for some positive functions $f,g$.
		Suppose that for every $R\geq R_0>1$, $T\geq T_0>0$
		\begin{equation}\label{e: stimavolumi-C1}
			\int_0^T f^{-\frac{1}{\sigma-1}}(t)\,dt\leq CT^{\delta_1},\qquad \sum_{x\in {B_R(x_0)}}\mu(x)g^{-\frac{1}{\sigma-1}}(x)\leq CR^{\delta_2},
		\end{equation}
		with $\delta_1,\delta_2\geq0$ such that $$(1+\alpha)\delta_1+\delta_2\leq\frac{(1+\alpha)\sigma}{\sigma-1},$$  and $\alpha, x_0$ as in \eqref{e7f}. Let $u\colon V\times [0,\infty)\to \mathbb{R}$ be a non-negative very weak solution of \eqref{e: mainequation}, then $u\equiv 0$.
	\end{corollary}

	\begin{corollary}\label{coroll2}
		Let assumption \eqref{e7f} be satisfied. Assume the graph $(V,\omega,\mu)$ is infinite. Let $\sigma>1$ and $v\colon V\times [0,\infty)\to \mathbb{R}$ be a positive function such that $v(x,t)\geq g(x)$ for every $(x,t)\in V\times[0,+\infty)$, for some positive function $g$.
		Suppose that for every $R\geq R_0>1$
		\begin{equation}\label{e: stimavolumi-C2}
			\sum_{x\in {B_R(x_0)}}\mu(x)g^{-\frac{1}{\sigma-1}}(x)\leq CR^{\frac{1+\alpha}{\sigma-1}},
		\end{equation}
		with $\alpha, x_0$ as in \eqref{e7f}. Let $u\colon V\times [0,\infty)\to \mathbb{R}$ be a non-negative very weak solution of \eqref{e: mainequation}, then $u\equiv 0$.
	\end{corollary}
	
	\begin{corollary}\label{coroll3}
		Let assumption \eqref{e7f} be satisfied. Assume the graph $(V,\omega,\mu)$ is infinite. Let $\sigma>1$ and assume that for every $R\geq R_0>1$
		\begin{equation}\label{e: stimavolumi-C3}
			\operatorname{Vol}(B_R(x_0))\leq CR^{\frac{1+\alpha}{\sigma-1}},
		\end{equation}
		with $\alpha, x_0$ as in \eqref{e7f}. Let $u\colon V\times [0,\infty)\to \mathbb{R}$ be a non-negative very weak solution of
		\[
		u_t \geq \Delta u + u^{\sigma} \quad \text{ in } V\times (0, +\infty),
		\]
		then $u\equiv 0$.
	\end{corollary}
	
	\medskip
	
	We postpone the much simpler case of a finite graph $V$ to Section \ref{finite}.

	\section{Proofs of the main results}\label{sec3}
	
	\begin{proof}[Proof of Theorem \ref{teo1}]\label{proofs}
		Let $\varphi\in C^2([0,+\infty))$ be a cut-off function such that $\varphi\equiv 1$ in $[0,1]$, $\varphi\equiv 0$ in $[2,+\infty)$, and $\varphi'\leq 0$. For each $x\in V$, $t\in [0,+\infty)$ and $R\geq R_0$ we define
		\[
		\psi_R(x,t)=\frac{t^{\theta_2}+ d(x_0,x)^{\theta_1}}{R^{\theta_1}},\qquad\phi_R(x,t)=\varphi(\psi_R(x,t)).
		\]
		
		We are going to prove the following two upper bounds
		\begin{enumerate}[(i)]
			\item\label{i: stima laplaciano}  there exists $C\geq 0$ such that for every $x\in V$ and $t\in [0,+\infty)$ it holds  \[-\Delta\phi_R(x,t)\leq \frac{C}{R^{1+\alpha}} \textbf{1}_{F_R}(x,t),\]
			where $F_R\coloneqq\{(x,t)\in V\times [0,+\infty)\colon (R/2)^{\theta_1}\leq d(x_0,x)^{\theta_1}+t^{\theta_2}\leq  (4R)^{\theta_1}\}$;
			\item\label{i: stima derivata in tempo} there exists $C\geq 0$ such that for every $x\in V$ and $t\in [0,+\infty)$ it holds  \[-\frac{\partial \phi_R}{\partial t}(x,t)\leq \frac{C}{R^{\frac{\theta_1}{\theta_2}}}\textbf{1}_{E_R}(x,t).\]
		\end{enumerate}
		
		Let us first prove \eqref{i: stima laplaciano}. Let $x\in V$ and $t\geq 0$. We observe that for each $p,r\geq 0$ the formula
		\begin{equation*}
			\varphi(p)=\varphi(r)+\varphi'(r)(p-r) +\frac{\varphi''(\xi)}{2}(p-r)^2,
		\end{equation*}
		holds for some $\xi\geq 0$ between $r$ and $p$. Letting $r=\psi_R(x,t)$ and $p=\psi_R(y,t)$, we get that for some $\xi$ between $\psi_R(x,t)$ and $\psi_R(y,t)$
		
		\begin{equation*}
			\begin{split}
				-\Delta\phi_R(x,t)&=-\frac{1}{\mu(x)}\sum_{y\sim x} \omega_{xy}[\varphi(\psi_R(y,t)) - \varphi(\psi_R(x,t))]\\
				&=-\frac{1}{\mu(x)}\varphi'(\psi_R(x,t))\sum_{y\sim x}\omega_{xy}(\psi_R(y,t) - \psi_R(x,t)) \\ &\quad- \frac{1}{2\mu(x)}\sum_{y\sim x}\omega_{xy} \varphi''(\xi) (\psi_R(y,t)-\psi_R(x,t))^2\\
				&=-\frac{1}{\mu(x)}\varphi'(\psi_R(x,t))\sum_{y\sim x}\omega_{xy}\left(\frac{d(x_0,y)^{\theta_1}-d(x_0,x)^{\theta_1}}{R^{\theta_1}}\right)\\
				&\quad - \frac{1}{2\mu(x)}\sum_{y\sim x}\omega_{xy} \varphi''(\xi) \left(\frac{d(x_0,y)^{\theta_1}-d(x_0,x)^{\theta_1}}{R^{\theta_1}}\right)^2.
			\end{split}
		\end{equation*}
		Similarly, for every $p,r\geq0$ we have
		\[
		p^{\theta_1}=r^{\theta_1}+\theta_1r^{\theta_1-1}(p-r)+\frac{\theta_1(\theta_1-1)}{2}\eta^{\theta_1-2}(p-r)^2
		\]
		and
		\[
		p^{\theta_1}=r^{\theta_1}+\theta_1\lambda^{\theta_1-1}(p-r)
		\]
		for some $\eta,\lambda\geq0$ between $r$ and $p$. Choosing $r=d(x_0,x)$ and $p=d(x_0,y)$ we get
		\begin{equation}\label{1}
			\begin{split}
				-\Delta\phi_R(x,t)&=-\frac{\theta_1 d(x_0,x)^{\theta_1-1}\varphi'(\psi_R(x,t))}{R^{\theta_1}}\Delta d(x_0,x)\\
				&\quad -\frac{\theta_1(\theta_1 - 1)}{2R^{\theta_1}}\varphi'(\psi_R(x,t))\frac{1}{\mu(x)}\sum_{y\sim x}\omega_{xy}\eta^{\theta_1 - 2}(d(x_0,y) -d(x_0,x))^2 \\
				&\quad - \frac{1}{2\mu(x)}\sum_{y\sim x}\omega_{xy} \varphi''(\xi) \left(\frac{\theta_1\lambda^{\theta_1-1}(d(x_0,y) -d(x_0,x))}{R^{\theta_1}}\right)^2.
			\end{split}
		\end{equation}
		for some $\eta$ and $\lambda$ between $d(x_0,x)$ and $d(x_0,y)$.
		Now, observe that $\varphi'(\psi_R(x,t))\leq 0$, $\varphi'$ is bounded and that $\varphi'(\psi_R(x,t))\equiv0$ on $E_R^c$. Then, by using \eqref{e7f}, we get for some $C>0$
		\begin{equation}\label{e: primotermine}
			\begin{split}
				-\frac{\theta_1 (d(x_0,x))^{\theta_1-1}\varphi'(\psi_R(x,t))}{R^{\theta_1}}\Delta d(x_0,x)&\leq  \frac{C}{R^{\theta_1}}(d(x_0,x))^{\theta_1-1-\alpha}\textbf{1}_{E_R}(x,t)\\
				&\leq  \frac{C}{R^{1+\alpha}}\textbf{1}_{E_R}(x,t),
			\end{split}
		\end{equation}
		where in the last inequality we have used that $\theta_1\geq2\geq1+\alpha$ and that $(x,t)\in E_R$ implies $d(x_0,x)\leq 2^{\frac{1}{\theta_1}}R$. Now note that
		\[
		|d(x_0,y)-d(x_0,x)|\leq d(x,y)\leq j
		\]
		and that
		\begin{equation}\label{2}
			0\leq\eta,\lambda\leq\max\{d(x_0,y),d(x_0,x)\}\leq d(x_0,x)+j.
		\end{equation}
		Then, also using \eqref{e7f}, if $R>j$ for every $(x,t)$ we have
		\begin{equation}\label{e: secondotermine}
			\begin{aligned}
				-\frac{\theta_1(\theta_1 - 1)}{2R^{\theta_1}}&\varphi'(\psi_R(x,t))\frac{1}{\mu(x)}\sum_{y\sim x}\omega_{xy}\eta^{\theta_1 - 2}(d(x_0,y) -d(x_0,x))^2\\
				& \leq \frac{C(d(x_0,x)+j)^{\theta_1 - 2}}{R^{\theta_1}}\textbf{1}_{E_R}(x,t)\frac{1}{\mu(x)}\sum_{y\sim x}\omega_{xy}\leq \frac{C}{R^2}\textbf{1}_{E_R}(x,t),
			\end{aligned}
		\end{equation}
		
		In order to estimate the last term on the right-hand side of \eqref{1}. we need the following
		
		\smallskip
		
		\emph{Claim}: let $R\geq2j$, suppose that either $d^{\theta_1}(x)+ t^{\theta_2}\geq (4R)^{\theta_1}$ or $d^{\theta_1}(x)+ t^{\theta_2}\leq (R/2)^{\theta_1}$ holds, then $\xi\notin (1,2)$.
		
		\smallskip
		
		Let us first study the case $d(x_0,x)\geq 2j$. Since we have
		\[
		\min\{\psi_R(x,t),\psi_R(y,t)\}\leq \xi\leq \max\{\psi_R(x,t),\psi_R(y,t)\},
		\]
		if $d(x_0,x)^{\theta_1}+ t^{\theta_2}\leq (R/2)^{\theta_1}$ we deduce
		\begin{equation*}
			\begin{split}
				\xi \leq \frac{(d(x_0,x)+j)^{\theta_1}+t^{\theta_2}}{R^{\theta_1}} &\leq 2^{\theta_1-1}\frac{d(x_0,x)^{\theta_1}+j^{\theta_1}+t^{\theta_2}}{R^{\theta_1}}\leq  2^{\theta_1-1}\frac{d(x_0,x)^{\theta_1}+t^{\theta_2}}{R^{\theta_1}}+\frac12\leq 1.
			\end{split}
		\end{equation*}
		On the other hand,  suppose that $d(x)^{\theta_1}+ t^{\theta_2}\geq (4R)^{\theta_1}$. Since the function $H(r)=\frac{(r-1)^{\theta_1}}{r^{\theta_1}-1}$ is increasing on $[2,\infty)$, for every
		$a,b\in \mathbb{R}$ such that $a\geq 2b\geq0$ we have
		\[
		(a-b)^{\theta_1}\geq \frac{1}{2^{\theta_1}}(a^{\theta_1}-b^{\theta_1}).
		\]
		Then we deduce that
		\begin{equation*}
			\begin{split}
				\xi &\geq \frac{(d(x_0,x)-j)^{\theta_1}+t^{\theta_2}}{R^{\theta_1}}\geq \frac{1}{2^{\theta_1}} \frac{d(x_0,x)^{\theta_1}-j^{\theta_1}+t^{\theta_2}}{R^{\theta_1}}\geq \frac{1}{2^{\theta_1}}\frac{d(x_0,x)^{\theta_1}+t^{\theta_2}}{R^{\theta_1}}-\frac{1}{4^{\theta_1}}\\
				& \geq 2^{\theta_1}-\frac{1}{4^{\theta_1}}\geq 2.
			\end{split}
		\end{equation*}
		Therefore we have proved that $\xi \notin (1,2)$, if $d(x_0,x)\geq 2j$. Assume now that $d(x_0,x)< 2j$. Similarly as in the previous case, if $d^{\theta_1}(x_0,x)+ t^{\theta_2}\leq (R/2)^{\theta_1}$ we have $\xi \leq 1$. On the other hand if $d^{\theta_1}(x_0,x)+t^{\theta_2}\geq (4R)^{\theta_1}$ we have
		\begin{equation*}
			\begin{split}
				\xi&\geq \frac{(d(x_0,x)-j)^{\theta_1}+t^{\theta_2}}{R^{\theta_1}}\geq \frac{t^{\theta_2}}{R^{\theta_1}}\geq \frac{(4R)^{\theta_1}-d(x_0,x)^{\theta_1}}{R^{\theta_1}}\\
				&\geq \frac{(4R)^{\theta_1}-(2j)^{\theta_1}}{R^{\theta_1}}\geq \frac{(4^{\theta_1}-1)R^{\theta_1}}{R^{\theta_1}}\geq 2.
			\end{split}
		\end{equation*}
		Also in this case $\xi\notin (1,2)$. This concludes the proof of the \emph{Claim}.
		
		\medskip
		
		We now observe that, by the \emph{Claim} proved above, if $R\geq2j$ and $(x,t)\in F_R^c$ then $\xi\notin(1,2)$ and hence $\varphi''(\xi)=0$. Thus, also using \eqref{2} and \eqref{e7f},  we deduce that if $R\geq2j$
		
		\begin{equation}\label{e: terzotermine}
			\begin{split}
				- \frac{1}{2\mu(x)}\sum_{y\sim x}&\omega_{xy} \varphi''(\xi) \left(\frac{\theta_1\lambda^{\theta_1-1}(d(x_0,y) -d(x_0,x))}{R^{\theta_1}}\right)^2\\
				&\leq C\frac{(d(x_0,x)+j)^{2\theta_1-2}}{R^{2\theta_1}}\textbf{1}_{F_R}(x,t)\frac{1}{\mu(x)}\sum_{y\sim x}\omega_{xy} \left(d(x_0,y) -d(x_0,x)\right)^2\\ &\leq \frac{C}{R^2}\textbf{1}_{F_R}(x,t),
			\end{split}
		\end{equation}
		since $(x,t)\in F_R$ and $R\geq2j$ imply $d(x_0,x)+j\leq 6R$.

		Combing \eqref{e: primotermine}, \eqref{e: secondotermine} and \eqref{e: terzotermine} with the facts that $E_R\subset F_R$ and $1+\alpha\leq 2$, we get
		\[
		-\Delta\phi_R(x,t)\leq \frac{C}{R^{1+\alpha}}\textbf{1}_{F_R}(x,t),
		\]
		which shows \eqref{i: stima laplaciano}.\\
		\smallskip
		
		Now, we are going to prove \eqref{i: stima derivata in tempo}. Let $x\in V$ and $t\geq 0$, we get
		
		\begin{equation*}
			\begin{split}
				-\frac{\partial \phi_R}{\partial t}(x,t)=-\varphi'(\psi_R(x,t))\theta_2 \frac{t^{\theta_2-1}}{R^{\theta_1}} \leq \frac{C}{R^{\theta_1}}R^{\theta_1-\frac{\theta_1}{\theta_2}}\textbf{1}_{E_R}(x,t)=\frac{C}{R^\frac{\theta_1}{\theta_2}}\textbf{1}_{E_R}(x,t),
			\end{split}
		\end{equation*}
		where we have used that $\varphi'(\psi_R(x,t))\equiv0$ on $E_R^c$, that $\varphi'$ is bounded and non-positive and that if $(x,t)\in E_R$, it holds that $t\leq CR^\frac{\theta_1}{\theta_2}$. This shows \eqref{i: stima derivata in tempo}.
		
		Now we are going to show that there exists a constant $C>0$ such that
		\begin{equation}\label{e:step2}
			\int_0^{+\infty}\sum_{x	\in V}\mu(x)u^{\sigma}(x,t)v(x,t)\leq C.
		\end{equation}
		
		Let us observe that the support of the function $\phi_R(x,t)$ is contained in $$Q_R\coloneqq B_{2^\frac{1}{\theta_1}R}(x_0)\times \big[0,2^{\frac{1}{\theta_2}}R^{\frac{\theta_1}{\theta_2}}\big],$$ so that all sums in $x\in V$ are finite, while all the integrals in the time-variable are on compact domains. Moreover $0\leq\phi_R\leq1$ in $V\times[0,+\infty)$ and $\phi_R(x,\cdot)\in C^1([0,+\infty))$ for all $x\in V$. Since $u$ is a nonnegative very weak solution of \eqref{e: mainequation}, by \eqref{e:48} testing the equation with $\phi_R^s$ and  $s>\frac{\sigma}{\sigma-1}$, we have
		\begin{equation}\label{3}
			\begin{split}
				\int_0^{\infty}\sum_{x\in V} \mu(x)v(x,t)u^\sigma(x,t)\phi^s_R(x,t)\, dt&\leq-s\int_0^{\infty} \sum_{x\in V}\mu(x)u(x,t)\phi^{s-1}_R(x,t)\frac{\partial\phi_R}{\partial t}(x,t)\,dt\\
				&\,\,\,\,\,\,-\int_0^{\infty} \sum_{x\in V}\mu(x)\Delta u(x,t)\phi^s_R(x,t)\,dt\\
				&\,\,\,\,\,\,-\sum_{x\in V}\mu(x)u(x,0)\phi^s_R(x,0)\\
				&\leq-s\int_0^{\infty} \sum_{x\in V}\mu(x)u(x,t)\phi^{s-1}_R(x,t)\frac{\partial\phi_R}{\partial t}(x,t)\,dt\\
				&\,\,\,\,\,\,-\int_0^{\infty} \sum_{x\in V}\mu(x)\Delta u(x,t)\phi^s_R(x,t)\,dt.
			\end{split}
		\end{equation}
		Considering the last right-hand side in inequality \eqref{3} term by term, by estimate \eqref{i: stima derivata in tempo} we have
		\begin{equation}\label{5}
			\begin{split}
				-s\int_0^{\infty} \sum_{x\in V}\mu(x)&u(x,t)\phi^{s-1}_R(x,t)\frac{\partial\phi_R}{\partial t}(x,t)\,dt\\
				&=-s\sum_{x\in V}\int_0^{\infty} \mu(x)u(x,t)\phi^{s-1}_R(x,t)\frac{\partial\phi_R}{\partial t}(x,t)\,dt\\
				&\leq \frac{C}{R^{\frac{\theta_1}{\theta_2}}} \sum_{x\in V}\mu(x)\int_{0}^{\infty}u(x,t)\phi^{s-1}_R(x,t)\textbf{1}_{E_R}(x,t)\, dt\\
				&\leq \frac{C}{R^{1+\alpha}} \int_{0}^{\infty}\sum_{x\in V}\mu(x)u(x,t)\phi^{s-1}_R(x,t)\textbf{1}_{E_R}(x,t)\, dt.
			\end{split}
		\end{equation}
		Let us now consider the second term on the right-hand side of \eqref{3}, by \eqref{e4f} we get
		\begin{equation*}
			-\int_0^{\infty}\sum_{x\in V}\mu(x)\Delta u(x,t)\phi^s_R(x,t)\, dt =-\int_0^{\infty}\sum_{x\in V} \mu(x) u(x,t)\Delta\phi_R^s(x,t)\,dt.
		\end{equation*}
		Now by convexity for every $p,r\geq0$ we have
		\[
		p^s\geq r^s+sr^{s-1}(p-r),
		\]
		hence, with $p=\phi_R(y,t)$, $r=\phi_R(x,t)$, we obtain
		\begin{equation*}
			\begin{split}
				-\Delta\phi_R^s(x,t)&=-\frac{1}{\mu(x)}\sum_{y\sim x}\omega_{xy}(\phi_R^s(y,t)-\phi_R^s(x,t))\\
				&\leq -\frac{1}{\mu(x)}\sum_{y\sim x}s\omega_{xy}\phi_R^{s-1}(x,t)(\phi_R(y,t)-\phi_R(x,t))\\
				&=-s\phi_R^{s-1}(x,t)\Delta\phi_R(x,t)
			\end{split}
		\end{equation*}
		Then, also using \eqref{i: stima laplaciano}, we deduce
		\begin{equation}\label{6}
			\begin{split}
				-\int_0^{\infty}\sum_{x\in V}\mu(x)&\Delta u(x,t)\phi^s_R(x,t)\, dt\\
				&\leq-\int_0^{\infty}\sum_{x\in V}s\mu(x)u(x,t)\phi_R^{s-1}(x,t)\Delta\phi_R(x,t) \,dt\\
				&\leq \frac{C}{R^{1+\alpha}} \int_0^{\infty}\sum_{x\in V}\mu(x)u(x,t)\phi_R^{s-1}(x,t)\textbf{1}_{F_R}(x,t) \,dt
			\end{split}
		\end{equation}
		Combining \eqref{3} with \eqref{5} and \eqref{6} and observing that $E_R\subset F_R$  we have
		\begin{equation}\label{42}
			\begin{split}
				\int_0^{\infty}\sum_{x\in V} \mu(x)&v(x,t)u^{\sigma}(x,t)\phi^s_R(x,t)\, dt\\
				&\leq  \frac{C}{R^{\alpha+1}}\int_0^{\infty}\sum_{x\in V}\mu(x) u(x,t)\phi_R^{s-1}(x,t)\textbf{1}_{F_R}(x,t)\,dt.
			\end{split}
		\end{equation}
		Then applying Young's inequality we obtain
		\begin{equation*}
			\begin{split}
				& \int_0^{\infty}\sum_{x\in V} \mu(x)v(x,t)u^{\sigma}(x,t)\phi^s_R(x,t)\, dt\\
				& \leq \frac{1}{\sigma}	\int_0^{\infty}\sum_{x\in V} \mu(x)v(x,t)u^{\sigma}(x,t)\phi^{s}_R(x,t)\textbf{1}_{F_R}(x,t)\,dt\\
				&\qquad + \frac{C}{R^{\frac{(1+\alpha)\sigma}{\sigma-1}}}\int_0^{\infty}\sum_{x\in V}\mu(x)\phi_R^{s-\frac{\sigma}{\sigma-1}}(x,t)v^{-\frac{1}{\sigma-1}}(x,t)\textbf{1}_{F_R}(x,t)\,dt\\
				&\leq \frac{1}{\sigma}	\int_0^{\infty}\sum_{x\in V} \mu(x)v(x,t)u^{\sigma}(x,t)\phi^{s}_R(x,t)\,dt\\
				&\qquad+ \frac{C}{R^{\frac{(1+\alpha)\sigma}{\sigma-1}}}	\int_0^{\infty}\sum_{x\in V}\mu(x)v^{-\frac{1}{\sigma-1}}(x,t)\textbf{1}_{F_R}(x,t)\,dt.
			\end{split}
		\end{equation*}
		Thus we have
		\begin{equation*}
			\begin{split}
				\int_0^{\infty}\sum_{x\in V} \mu(x)v(x,t)u^{\sigma}(x,t)\phi^s_R(x,t) dt\leq \frac{C}{R^{\frac{(1+\alpha)\sigma}{\sigma-1}}}	\int_0^{\infty}\sum_{x\in V}\mu(x)v^{-\frac{1}{\sigma-1}}(x,t)\textbf{1}_{F_R}(x,t)dt.
			\end{split}
		\end{equation*}
		Now note that $\phi_R\geq0$, $\phi_R\equiv 1$ on $B_R\coloneq \{(x,t)\in V\times [0,+\infty)\colon d(x_0,x)^{\theta_1}+t^{\theta_2}\leq R^{\theta_1}\}$ and moreover that, if $m\in\mathbb{N}$ is the first positive integer such that $m\geq3\theta_1-1$, then for every $R>0$ we have
		\[
		F_R\subset\bigcup_{k=0}^m E_{2^{\frac{k}{\theta_1}-1}R}.
		\]
		Then for every large enough $R>0$ we have
		\begin{equation*}
			\begin{split}
				\int_0^{\infty}\sum_{x\in V}\mu(x)&v(x,t)u^{\sigma}(x,t)\textbf{1}_{B_R}(x,t)\,dt\\
				&\leq \int_0^{\infty}\sum_{x\in V}\mu(x)v(x,t)u^{\sigma}(x,t)\phi^s_R(x,t)\,dt\\
				&\leq  \frac{C}{R^{\frac{(1+\alpha)\sigma}{\sigma-1}}}	\int_0^{\infty}\sum_{x\in V}\mu(x)v^{-\frac{1}{\sigma-1}}(x,t)\textbf{1}_{F_R}(x,t)\,dt \\
				&\leq\frac{C}{R^{\frac{(1+\alpha)\sigma}{\sigma-1}}}\sum_{k=0}^m\int_0^{\infty}\sum_{x\in V}\mu(x)v^{-\frac{1}{\sigma-1}}(x,t)\textbf{1}_{E_{2^{\frac{k}{\theta_1}-1}R}}(x,t)\,dt\\
				&\leq  \frac{C}{R^{\frac{(1+\alpha)\sigma}{\sigma-1}}}\sum_{k=0}^m \left(2^{\frac{k}{\theta_1}-1}R\right)^{\frac{(1+\alpha)\sigma}{\sigma-1}}
				\leq C,
			\end{split}
		\end{equation*}
		by \eqref{e: stimavolumi}. Since this last inequality does not depend on $R$, passing to the limit as $R$ tends to $+\infty$ we get
		\begin{equation}\label{43}
			\int_0^{\infty}\sum_{x\in V}\mu(x)v(x,t)u^{\sigma}(x,t)dt\leq  C,
		\end{equation}
		which proves \eqref{e:step2}. It remains now to show that $u\equiv 0$ on $V\times [0,+\infty)$. In order to do so, let us consider again inequality \eqref{42}. By a standard application of the H\"older inequality and using condition \eqref{e: stimavolumi} as above, we get
		\begin{equation*}
			\begin{split}
				\int_0^{+\infty}\sum_{x\in V}\mu(x)&v(x,t)u^\sigma(x,t)\textbf{1}_{B_R}(x,t)\,dt\\
				&\leq\int_0^{\infty}\sum_{x\in V} \mu(x)u^{\sigma}(x,t)v(x,t)\phi^s_R(x,t) \,dt \\
				&\leq \frac{C}{R^{\alpha+1}}\left(\int_0^{\infty}\sum_{x\in V}\mu(x)u^{\sigma}(x,t)v(x,t)\phi^s_R(x,t)\textbf{1}_{F_R}(x,t)\,dt\right)^{\frac{1}{\sigma}}\\
				&\qquad \left(\int_0^{\infty}\sum_{v\in V}\mu(x)v^{-\frac{1}{\sigma-1}}(x,t)\phi^{s-\frac{\sigma}{\sigma-1}}_R(x,t)\textbf{1}_{F_R}(x,t)\right)^{\frac{\sigma-1}{\sigma}} \\&\leq C\left(\int_0^{\infty}\sum_{x\in V}\mu(x)v(x,t) u^{\sigma}(x,t)\textbf{1}_{F_R}(x,t)\,dt\right)^{\frac{1}{\sigma}}.
			\end{split}
		\end{equation*}
		By \eqref{43} the last right-hand side of the above inequality tends to zero as $R$ tends to $+\infty$. It follows that
		\begin{equation*}
			\int_0^{+\infty}\sum_{x\in V}\mu(x)v(x,t)u^\sigma(x,t)\,dt\leq 0.
		\end{equation*}
		Since $\mu$ and $v$ are positive, we get $u\equiv 0$ in $V\times [0,+\infty)$, and the proof of the theorem is complete.
	\end{proof}

	In the following we derive the corollaries of our main result. 
	
	\begin{proof}[Proof of Corollary \ref{coroll1}]
		Let $\theta_1=2(1+\alpha)$, $\theta_2=2$, then the set $E_R$ defined in \eqref{22} satisfies
		\[
		E_R\subset B_{2^\frac{1}{\theta_1}R}(x_0)\times\big[0,2^\frac{1}{\theta_2}R^{1+\alpha}\big].
		\]
		Hence we see that for every large enough $R>0$ we have
		\begin{equation*}
			\begin{split}
				\int_0^{+\infty}\sum_{x\in V} \textbf{1}_{E_R}(x,t)& v(x,t)^{-\frac{1}{\sigma-1}}\mu(x)dt\\
				&\leq \int_0^{+\infty}\sum_{x\in V} \textbf{1}_{E_R}(x,t) f^{-\frac{1}{\sigma-1}}(t)g^{-\frac{1}{\sigma-1}}(x)\mu(x)dt\\
				&\leq \left(\int_0^{2^\frac{1}{\theta_2}R^{1+\alpha}}f^{-\frac{1}{\sigma-1}}(t)\,dt\right)\left(\sum_{x\in V}\mu(x)g^{-\frac{1}{\sigma-1}}(x)\textbf{1}_{B_{2^\frac{1}{\theta_1}R}(x_0)}\right)\\
				&\leq CR^{(1+\alpha)\delta_1+\delta_2}\leq CR^\frac{(1+\alpha)\sigma}{\sigma-1}.
			\end{split}
		\end{equation*}
		Then $u\equiv0$ by Theorem \ref{teo1}.
	\end{proof}

	\begin{proof}[Proof of Corollary \ref{coroll2}]
		This is an immediate consequence of Corollary \ref{coroll1}, choosing $f\equiv1$.
	\end{proof}

	\begin{proof}[Proof of Corollary \ref{coroll3}]
		This is an immediate consequence of Corollary \ref{coroll2}, choosing $g\equiv1$.
	\end{proof}

	\section{Examples}\label{ex}
	In the following example we deal with the special graph $V=\mathbb Z^N$, $N\geq1$, and with $v\equiv 1$, where our result is optimal.
	
	\begin{example}\label{ex1}
		The \textit{N-dimensional integer lattice graph} is the graph consisting of the set of vertices $\mathbb{Z}^N=\{x=(x_1,\ldots,x_N)\,:\,x_k\in\mathbb{Z}\text{ for }k=1,\ldots,N\}$, the edge weight
		\[\omega:\mathbb Z^N\times \mathbb Z^N\to [0,+\infty),\]
		\begin{equation*}
			\omega_{xy}= \begin{cases}
				1, \,\, \text{ if } \sum_{k=1}^N|y_k-x_k|=1,\\
				0, \,\, \text{ otherwise}
			\end{cases}
		\end{equation*}
		for every $x,y\in \mathbb{Z}^N$ and the node measure $$\mu(x)\coloneqq \sum_{y\sim x}\omega_{xy}=2N$$ for every $x\in \mathbb{Z}^N$.
		Then we see that the set of edges is
		\begin{equation*}
			E=\left\{(x,y)\colon x,y \in \mathbb{Z}^N,\,\, \sum_{i=1}^{N}|x_i-y_i|=1\right\}.
		\end{equation*}
		For every function $u:\mathbb{Z}^N\rightarrow\mathbb{R}$ we have $$\Delta u(x)=\frac{1}{2N}\sum_{y\sim x}(u(y)-u(x))\qquad \text{ for all }x\in\mathbb{Z}^N.$$
		
		We consider the graph $(\mathbb{Z}^N,\omega, \mu)$ endowed with the euclidean distance
		\begin{equation*}
			|x-y|\coloneqq d(x,y)=\left(\sum_{i=1}^N(x_i-y_i)^2\right)^{1/2}.
		\end{equation*}
		It is immediate to see that for $R>1$ we have
		\begin{equation*}
			\sum_{x\in B_R(0)}\mu(x)= 2N\sum_{x\in B_{R}(0)}1\leq C R^{N}.
		\end{equation*}
		We also have that condition \eqref{e7f}-(v) holds, with $\alpha=1$. Indeed it is not difficult to see that
		\begin{equation*}
			\begin{split}
				\Delta d^2(x,0) &= \sum_{y\sim x}\frac{\omega_{xy}}{\mu(x)}(d^2(y,0)-d^2(x,0))=1
			\end{split}
		\end{equation*}
		for all $x\in\mathbb{Z}^N$, see e.g. Example 5.1 in \cite{MPS1}. By convexity of the real valued function $h(p)=p^2$ we conclude that
		\begin{equation*}
			\Delta d(x,0) \leq \frac{1}{2 d(x,0)}
		\end{equation*}
		if $d(x,0)\geq1$, see Remark \ref{rem111}. Then Theorem \ref{teo1} and all its Corollaries apply, with $\alpha=1$, if the corresponding weighted volume growth assumption is satisfied. In particular, by Corollary \ref{coroll3}, we see that
		\[
		u_t \geq \Delta u + u^{\sigma} \quad \text{ in } \mathbb{Z}^N\times (0, +\infty)
		\]
		does not have nontrivial nonnegative solutions if $N\leq\frac{2}{\sigma-1}$, that is if $1<\sigma\leq1+\frac{2}{N}$.
		
		From the global existence result in \cite{LW2}, equation 
		\[u_t = \Delta u + u^{\sigma} \quad \text{ in } \mathbb{Z}^N\times (0, +\infty)\]
		completed with the initial condition $u(x,0)=u_0(x)\geq0$, $\forall x\in \mathbb Z^N$,
		admits a global in time nonnegative solution, provided that $u_0$ is small enough and $\sigma>1+\frac{2}{N}.$ Therefore, our nonexistence result is optimal.  
	\end{example}
	
	In the next two examples we apply our general theorem to a couple of specific graphs.
	
	\begin{example}
		Let $(V,\omega,\mu)$ be a connected, locally finite, undirected, weighted graph such that $\sum_{y\sim x}\omega_{xy}\leq C \mu(x)$ for every $x\in V$. Then condition \eqref{e7f} is satisfied with $\alpha=0$ by the natural distance $d_*$ on $V$, defined in \eqref{333}, see Remark \ref{4444}. Then Theorem \ref{teo1} and all its Corollaries apply, with $\alpha=0$. In particular by Corollary \ref{coroll3} we see that
		\[
		u_t \geq \Delta u + u^{\sigma} \quad \text{ in } V\times (0, +\infty)
		\]
		does not have nontrivial nonnegative solutions if for some $x_0\in V$ and every $R\geq R_0>0$ we have $$\operatorname{Vol}(B_R(x_0))\leq CR^{\frac{1}{\sigma-1}}.$$
	\end{example}
	
	\begin{example}
		Let $(V_1,\omega_1,\mu_1)$ and $(V_2,\omega_2,\mu_2)$ be connected, locally finite, undirected, weighted graphs, each endowed with a pseudometric $d_1$,$ d_2$, respectively, with finite jump size $j_1$, $j_2$. Define $(V,\omega,\mu)$ by setting $V=V_1\times V_2$ and for every $(x_1,x_2),(y_1,y_2)\in V_1\times V_2$ let
		\begin{equation}\label{777}
			\omega((x_1,x_2),(y_1,y_2)):=\begin{cases}
				\omega_1(x_1,y_1)&\text{ if }x_2=y_2\\
				\omega_2(x_2,y_2)&\text{ if }x_1=y_1\\
				0&\text{ otherwise.}
			\end{cases}
		\end{equation}
		Let $\mu$ be any positive measure on $V_1\times V_2$ such that
		\begin{equation}\label{779}
			\mu(x_1,x_2)\geq C\max\{\mu_1(x_1),\mu_2(x_2)\}
		\end{equation}
		for every $(x_1,x_2)\in V_1\times V_2$. Then $(V,\omega,\mu)$ is a connected, locally finite, undirected, weighted graph. Moreover, for any $p\in[1,2]$ let
		$$d((x_1,x_2),(y_1,y_2)):=\left(d_1^p(x_1,y_1)+d_2^p(x_2,y_2)\right)^\frac{1}{p}$$
		for every $(x_1,x_2),(y_1,y_2)\in V_1\times V_2$. Then $d$ is a pseudometric on $V$, with finite jump size $j\leq\max\{j_1,j_2\}$. Let $(w_1,w_2)\in V_1\times V_2$ be fixed, using \eqref{777} we have
		\begin{equation}\label{778}
			\begin{split}
				\Delta_V& d^p((x_1,x_2),(w_1,w_2))\\
				&=\frac{1}{\mu(x_1,x_2)}\sum_{(y_1,y_2)\in V}\omega((x_1,x_2),(y_1,y_2))(d^p((y_1,y_2),(w_1,w_2))-d^p((x_1,x_2),(w_1,w_2)))\\
				&=\frac{1}{\mu(x_1,x_2)}\sum_{(y_1,y_2)\in V}\omega((x_1,x_2),(y_1,y_2))(d_1^p(y_1,w_1)+d_2^p(y_2,w_2)-d_1^p(x_1,w_1)-d_2^p(x_2,w_2))\\
				&=\frac{1}{\mu(x_1,x_2)}\sum_{(x_1,y_2)\in V}\omega_2(x_2,y_2)(d_2^p(y_2,w_2)-d_2^p(x_2,w_2))\\
				&\qquad+\frac{1}{\mu(x_1,x_2)}\sum_{(y_1,x_2)\in V}\omega_1(x_1,y_1)(d_1^p(y_1,w_1)-d_1^p(x_1,w_1))\\
				&=\frac{\mu_2(x_2)}{\mu(x_1,x_2)}\Delta_{V_2} d_2^p(x_2,w_2)+\frac{\mu_1(x_1)}{\mu(x_1,x_2)}\Delta_{V_1} d_1^p(x_1,w_1)
			\end{split}
		\end{equation}
		
		\medskip
		
		Now suppose $(V_1,\omega_1,\mu_1)$ and $(V_2,\omega_2,\mu_2)$ are infinite graphs that satisfy condition \eqref{e7f} for some $\alpha_1,\alpha_2\in[0,1]$, with respect to some points $w_1\in V_1$ and $w_2\in V_2$,  respectively. Then both graphs satisfy the same condition for $\alpha=\min\{\alpha_1,\alpha_2\}$, and by Remark \ref{rem111} we have
		\begin{equation}\label{780}
			\Delta_{V_1} d_1^{1+\alpha}(x_1,w_1)\leq C, \qquad\Delta_{V_2} d_2^{1+\alpha}(x_2,w_2)\leq C,
		\end{equation}
		for every $x_1\in V_1$, $x_2\in V_2$ such that $d_1(x_1,w_1)\geq R_0$, $d_2(x_2,w_2)\geq R_0$. Since the sets $\{x_1\in V_1\,:\,d_1(x_1,w_1)\leq R_0\}$ and $\{x_2\in V_2\,:\,d_2(x_2,w_2)\leq R_0\}$ are finite by condition \eqref{e7f}, up to further increasing the constant $C>0$ we have that \eqref{780} holds for every $x_1\in V_1$, $x_2\in V_2$.
		
		By \eqref{778} with $p=1+\alpha$, \eqref{779} and \eqref{780} then we have
		\begin{equation}\label{781}
			\Delta_V d^{1+\alpha}((x_1,x_2),(w_1,w_2))\leq C,
		\end{equation}
		for every $(x_1,x_2)\in V$. Then by Remark \ref{rem111} we deduce that
		\begin{equation}\label{782}
			\Delta_V d((x_1,x_2),(w_1,w_2))\leq \frac{C}{d^\alpha((x_1,x_2),(w_1,w_2))},
		\end{equation}
		for every $(x_1,x_2)\in V$ such that $d((x_1,x_2),(w_1,w_2))\geq1$. Moreover, since $V_1,V_2$ both satisfy condition \eqref{e7f}, it is easy to see that the set $\{(y_1,y_2)\in V\colon d((x_1,x_2),(y_1,y_2))\leq R\}$ is finite for every $R>0$ and every $(x_1,x_2)\in V$.
		
		We have thus shown that if $(V_1,\omega_1,\mu_1)$ and $(V_2,\omega_2,\mu_2)$ are infinite graphs that satisfy condition \eqref{e7f} for some $\alpha_1,\alpha_2\in[0,1]$, then $(V,\omega,\mu)$ is an infinite graph that satisfy condition \eqref{e7f} for $\alpha=\min\{\alpha_1,\alpha_2\}$.
		
		\medskip
		
		On the other hand, now suppose $(V_1,\omega_1,\mu_1)$ is an infinite graph that satisfy condition \eqref{e7f} for some $\alpha\in[0,1]$, and assume $(V_2,\omega_2,\mu_2)$ is a finite graph. Then, as before,
		\begin{equation*}
			\Delta_{V_1} d_1^{1+\alpha}(x_1,w_1)\leq C
		\end{equation*}
		for every $x_1\in V_1$. Moreover, since $V_2$ is finite, it is obvious that
		\begin{equation*}
			\Delta_{V_2} d_2^{1+\alpha}(x_2,w_2)\leq C,
		\end{equation*}
		for every $x_2\in V_2$. We conclude also in this case that \eqref{781}, and thus \eqref{782}, hold. Finally we note that the set $\{(y_1,y_2)\in V\colon d((x_1,x_2),(y_1,y_2))\leq R\}$ is finite for every $R>0$ and every $(x_1,x_2)\in V$.
		
		We have then shown that if $(V_1,\omega_1,\mu_1)$ is an infinite graph that satisfy condition \eqref{e7f} for some $\alpha\in[0,1]$ and if $(V_2,\omega_2,\mu_2)$ is a finite graph, then $(V,\omega,\mu)$ is an infinite graph that satisfies condition \eqref{e7f}, for the same value $\alpha$.
		
		A particular example is given by the choice $V_1=\mathbb{Z}^N$, see Example \ref{ex1}, and $V_2=(\mathbb{Z}_m)^K$, with $N,K\geq1$, $m\geq2$. Here $\mathbb{Z}_m=\{[0],[1],\ldots,[m-1]\}$, with $$\omega([a],[b])=\begin{cases}1&\text{ if }[b]=[a]\pm[1]\\0&\text{otherwise}\end{cases},$$
		with
		$$\mu([a])=\sum_{[b]\in\mathbb{Z}_m}\omega([a],[b])=2$$
		and with
		$$
		d_*([a],[b])=\min\{n\in\mathbb{N}\colon\exists\text{ path of }n+1\text{ nodes connecting }[a],[b]\},
		$$
		for every $[a],[b]\in\mathbb{Z}_m$.

		Since it is easy to see that for $R>0$ large enough on $V=\mathbb{Z}^N\times(\mathbb{Z}_m)^K$ one has
		$$\operatorname{Vol}(B_R((0,[0])))\asymp R^N,$$
		by Corollary \ref{coroll3} we have that
		\[
		u_t \geq \Delta u + u^{\sigma} \quad \text{ in } \mathbb{Z}^N\times(\mathbb{Z}_m)^K\times (0, +\infty)
		\]
		does not have any nontrivial nonnegative solution if $1<\sigma\leq1+\frac{2}{N}$, see also Example \ref{ex1}.
	\end{example}
	
	\section{Finite graphs}\label{finite}
	In this Section, the main difference compared with the preceding ones where we considered only {\it infinite} graphs, is that we deal with {\it finite} graphs. Observe that in this framework condition \eqref{e7f} is automatically fulfilled; on the other hand, the weighted volume growth assumption \eqref{e: stimavolumi} takes a simpler form (see \eqref{e: stimavolumi2} below).
	
	\begin{remark}
		Note that there is no topological definition of boundary of a graph. On the other hand, if $V$ is a graph and $G\subset V$ is a subgraph, then one can define the boundary of $G$ relative to $V$, that is $$\partial G=\{x\in G\,:\, \exists\, y\in V\setminus G, y\sim x\}.$$
		In particular, in Theorem \ref{teo2} we consider finite graphs which thus have no boundary. This is the counterpart in the graph setting of a compact Riemannian manifold without boundary.
		
		Let us mention that if equation \eqref{e: mainequation} is posed on a finite subgraph $G\subset V$, completed with appropriate boundary conditions, then one expects different results. This second case would be the analog of considering a compact Riemannian manifold with boundary.
	\end{remark}
	
	In this context our result reads as follows. 
	\begin{theorem}\label{teo2}
		Let $(V,\omega,\mu)$ be a weighted \emph{finite} graph. Let $v\colon V\times [0,\infty)\to \mathbb{R}$ be a positive function, $\sigma>1$. Suppose that for every $T\geq T_0>0$
		\begin{equation}\label{e: stimavolumi2}
			\int_T^{2T}\sum_{x\in V} v^{-\frac{1}{\sigma-1}}(x,t)\mu(x)\,dt\leq CT^{\frac{\sigma}{\sigma-1}}.
		\end{equation}
		Let $u\colon V\times [0,\infty)\to \mathbb{R}$ be a non-negative very weak solution of \eqref{e: mainequation}, then $u\equiv 0$.
	\end{theorem}
	\begin{proof}
		The proof of Theorem \ref{teo2} is similar to that of Theorem \ref{teo1}, therefore we will only give a sketch of it. Let $\varphi\in C^2([0,+\infty))$ be a cut-off function such that $\varphi\equiv 1$ in $[0,1]$, $\varphi\equiv 0$ in $[2,+\infty)$, and $\varphi'\leq 0$. We use the definition of very weak solution \eqref{e:48}, with the test function $\phi^s_{T}(t):=\varphi^s(\frac{t}{T})$ for some $s>\frac{\sigma}{\sigma-1}$ and any $T>0$, which is a feasible test function since $V$ is finite. We note that by the integration by parts formula \eqref{e4f} we have $$\sum_{x\in V}\mu(x)\Delta u(x,t) \phi^s_T(t)=\sum_{x\in V}\mu(x) u(x,t)\Delta \phi^s_T(t)=0$$ for every $t$. Then we obtain
		\begin{equation}\label{999}
			\begin{split}
				\int_0^{+\infty}\sum_{x\in V}\mu(x)v(x,t)u^\sigma(x,t)\phi^s_T(t)\,dt&\leq -s\int_0^{+\infty}\sum_{x\in V}\mu(x)u(x,t)\phi_T^{s-1}(t)\frac{\partial\phi_T}{\partial t}(t)\,dt\\
				&\leq \frac{C}{T}\int_T^{2T}\sum_{x\in V}\mu(x)u(x,t)\phi_T^{s-1}(t)\,dt
			\end{split}
		\end{equation}
		By Young's inequality one obtains
		\begin{equation*}
			\begin{split}
				\int_0^T\sum_{x\in V}\mu(x)v(x,t)u^\sigma(x,t)\,dt&\leq
				\int_0^{+\infty}\sum_{x\in V}\mu(x)v(x,t)u^\sigma(x,t)\phi^s_T(t)\,dt\\
				&\leq \frac{C}{T^{\frac{\sigma}{\sigma-1}}}\int_T^{2T}\sum_{x\in V}\mu(x)v^{-\frac{1}{\sigma-1}}(x,t)\,dt\leq C,
			\end{split}
		\end{equation*}
		by \eqref{e: stimavolumi2}. Then we conclude that 
		\begin{equation}\label{000}
			\int_0^{+\infty}\sum_{x\in V}\mu(x)v(x,t)u^\sigma(x,t)\,dt<+\infty.
		\end{equation}
		Starting again from \eqref{999}, from the H\"older inequality and using condition \eqref{e: stimavolumi2} as above, we get
		\begin{equation*}
			\begin{split}
				\int_0^T\sum_{x\in V}\mu(x)v(x,t)u^\sigma(x,t)\,dt&\leq
				\int_0^{+\infty}\sum_{x\in V}\mu(x)v(x,t)u^\sigma(x,t)\phi^s_T(t)\,dt\\
				&\leq C\left(\int_T^{2T}\sum_{x\in V}\mu(x)v(x,t)u^\sigma(x,t)\phi^s_T(t)\,dt\right)^{\frac{1}{\sigma}}\\
				&\leq C\left(\int_T^{2T}\sum_{x\in V}\mu(x)v(x,t)u^\sigma(x,t)\,dt\right)^{\frac{1}{\sigma}}.
			\end{split}
		\end{equation*}
		By letting $T$ tend to $+\infty$, from \eqref{000} we conclude that $$\int_0^{+\infty}\sum_{x\in V}\mu(x)v(x,t)u^\sigma(x,t)\,dt=0,$$ and hence $u\equiv0$ on $V\times[0,+\infty)$.
	\end{proof}
	
	In the following corollary we deal with the case of a potential $v$ independent of $t$ (see also \cite{Wu} for a preceding result in the case $v\equiv 1$).
	
	\begin{corollary}\label{coroll4}
		Let $(V,\omega,\mu)$ be a weighted \emph{finite} graph. Let $v\colon V\to \mathbb{R}$ be a positive function, $\sigma>1$.
		Let $u\colon V\times [0,\infty)\to \mathbb{R}$ be a non-negative very weak solution of \eqref{e: mainequation}, then $u\equiv 0$.
	\end{corollary}
	\begin{proof}
		This is a straightforward consequence of Theorem \ref{teo2}, since we have
		\begin{equation*}
			\int_T^{2T}\sum_{x\in V} v^{-\frac{1}{\sigma-1}}(x)\mu(x)\,dt=AT\leq AT^{\frac{\sigma}{\sigma-1}}
		\end{equation*}
		for some constant $A>0$, for every $T\geq1$.
	\end{proof}

	%-------------------------------------------------------%
	%                                                       %
	% 						BIBLIOGRAPHY    				%
	%                                                       %
	%-------------------------------------------------------%

\end{document}